\documentclass[11pt]{article}
\usepackage[utf8]{inputenc}
\usepackage{graphicx, color}
\usepackage{mymacro}
\usepackage{wasysym}
\usepackage[all]{xy}
\usepackage{bm}
\usepackage{typearea}
\usepackage{cleveref}
\renewcommand{\Sh}{\mathrm{Sh}}

\renewcommand{\Mod}{\mathrm{Mod}}

\newcommand{\musupp}{\mu\mathrm{supp}}

\newcommand{\Int}{\mathrm{Int}}
\newcommand{\CSh}{\mathrm{CSh}}

\newcommand{\comp}{\mathrm{c}}
\newcommand{\SQ}{\mathrm{SQ}}
\newcommand{\coim}{\mathrm{coim}}
\newcommand{\almost}{\mathrm{a}}

\title{Almost equivalences between \\ Tamarkin category and Novikov sheaves}
\author{Tatsuki Kuwagaki\footnote{Department of Mathematics, Kyoto University \\
Kitashirakawa Oiwakecho, Sakyo-ku, Kyoto 606-8502, Japan\\
Email: {tatsuki.kuwagaki@gmail.com}\\}}
\date{\today}

\begin{document}

\maketitle
\begin{abstract}
    We revisit the relationship between the Tamarkin's extra variable $\bR_t$ and Novikov rings. We prove that the equivariant version of Tamarkin category is almost equivalent (in the sense of almost mathematics) to the category of derived complete modules over the Novikov ring.
\end{abstract}
\section{Introduction}
\subsection{Main result}
In symplectic geometry, there exists a series of extra variables, which are (expected to be) just various aspects of the same variable.

\begin{enumerate}
    \item In microlocal sheaf theory, it is called $t$ and introduced by Tamarkin~\cite{Tam} to study sheaf quantizations, which was also envisioned by Sato in relation with WKB analysis~\cite{Sato}.
    \item In deformation quantization/WKB analysis/twistor theory, it is the Laplace dual of the inverse of $\hbar$. 
    \item In Floer theory, it is the exponent/valuation of the universal Novikov ring. 
\end{enumerate}

For example, in \cite{WKBkuw}, we discussed the relationship between (1) and (2); By using exact WKB analysis for $\hbar$-differential operators, we construct sheaf quantization. In the process, the value in (2) is identified with (1). Also, some hints of the relationship between (1) and (3) can be seen from the work of Tamarkin and \cite{AI, IK}; In \cite{AI}, the relation between symplectic energy and $t$ is explained, which is the same role as the Novikov variable. In \cite{IK}, this point of view is further expanded to develop a sheaf-theoretic construction parallel to Lagrangian Floer theory.

In this paper, we investigate the relationship between (1) and (3) further. 

We would like to formulate our main theorem now. Let $\bK$ be a field. Let $X$ be a manifold and $\bR_t$ be the real line. We denote the discrete additive group of the real numbers by $\bR_{d}$. We then have the equivariant derived category $\Sh^{\bR_d}(X\times \bR_t,\bK)$ of $\bK$-modules sheaves. We quotient this category by the sheaves with non-positive microsupport and denote it by $\Sh^{\bR_d}_{>0}(X\times \bR_t,\bK)$. The non-equivariant version of this category was originally introduced by Tamarkin~\cite{Tam}, and the equivariant version was later introduced by the present author~\cite{WKBkuw}. On the other hand, we have the universal Novikov ring $\Lambda_0$ over $\bK$. We denote the derived category of $\Lambda_0$-module sheaves by $\Sh(X, \Lambda_0)$.
\begin{theorem}
We have an almost embedding 
\begin{equation}
    \frakA \colon \Sh^{\bR_d}_{>0}(X\times \bR_t,\bK)\hookrightarrow_\almost\Sh(X,\Lambda_0).
\end{equation}
The almost images are sheaves valued in derived complete modules over $\Lambda_0$.
\end{theorem}
Here the term ``almost" comes from the almost mathematics~\cite{GabberRamero}. Roughly speaking, the above theorem without ``almost" holds after negating almost zero modules. The precise meaning will be explained in the body of the paper. The use of the almost mathematics in symplectic geometry also happened in \cite{fukaya2021gromovhausdorff}.

In the body of paper, we will state a more generalized version of the theorem: Namely, for sheaves equivariant with respect to a subgroup $\bG\subset \bR$. In this version, the right hand side is a sheaf valued in modules over a certain completion of infinite versions of $A_n$-quiver algebra. For example, in the classical case, it is a certain completion of the poset $(\bR, \geq )$ used in the literature of persistent modules.

\begin{remark}
    Here let us introduce some previous works.
    \begin{enumerate}
        \item In \cite{WKBkuw}, we found the Novikov ring a posteriori after the construction of the equivariant version of Tamarkin's category. Namely, we showed that the said category is linear over the Novikov ring. In this paper, we show that Novikov sheaves and the equivariant Tamarkin category is actually (almost) equivalent, which is stronger that the former one. In particular, this equivalence allows us to have an alternative model for the category of sheaf quantizations.
        \item The holonomic Riemann--Hilbert correspondence by D'Agnolo--Kashiwara~\cite{DK} is formulated and proved by using a variant of Tamarkin category, called enhanced ind-sheaves. In \cite{IrregPerv}, we use the sheaves valued in modules over the finite Novikov ring $\bK[\bR_{\geq 0}]$ to reformulate the holonomic Riemann--Hilbert correspondence. One can reprove this result by using this paper with some understanding on the relationship between enhanced ind-sheaves and Tamarkin category. This will be further explained in \cite{KudomiKuwagaki}.
        \item After completing this manuscript, Bingyu Zhang informed the author that Vaintrob sketched a similar result for non-negatively microsupported sheaves in his unpublished manuscript~\cite{Vaintrob}.
    \end{enumerate}
\end{remark}

\subsection{Why is the theorem important?}
In this section, we explain the reason why the stated main theorem seems to be important.
\subsubsection{Symplectic geometry}
The original non-equivariant version of Tamarkin's category~\cite{Tam} is known to capture many symplectic geometric information of $T^*X$ and its Lagrangian submanifolds. Viterbo~\cite{Viterbo} envisioned that the Tamarkin category contains filtered Fukaya category of exact Lagrangians of $T^*X$, and constructed a baby version of the functor (a sheaf-theoretic construction of the expected images were constructed by Guillermou~\cite{Guillermou}). The idea is very canonical: Consider the family Floer functor with respect to the cotangent fibers, then one gets a sheaf on $X$. Since the Floer cohomology is filtered in the exact case, one can translate it to the additional $\bR$-direction of the Tamarkin category.

If one expects a generalization of the family Floer functor to the non-exact case, the functor should take its value in the category of the sheaves over the Novikov ring, since the Floer theory is defined over the Novikov ring~\cite{FOOO}, due to the infiniteness of the number of holomorphic disks.

On the other hand, in the non-exact case, the works of \cite{WKBkuw, IK} show that the equivariant Tamarkin category works well to treat non-exact Lagrangians. So, in particular, \cite{IK} conjectures that a version of non-exact Fukaya category of $T^*X$ is embedded to $\Sh^{\bR_d}_{>0}(X\times \bR_t,\bK)$. An integral version of the conjecture is proved in ~\cite{KPS} without using family Floer theory, but using \cite{GPS}.

To prove this conjecture by using the family Floer theory, one has to relate $\Sh^{\bR_d}_{>0}(X\times \bR_t,\bK)$ and the category of sheaves over the Novikov ring, which is the natural target of the family Floer functor. The main theorem explains the relationship precisely.

Aside from this application, we will explain two applications of (the philosophy) of the theorem. 

The first one is a first step toward nonconic microlocal sheaf theory of sheaves valued in modules over real valuation rings. Namely, for a manifold $X$, a real valuation ring $R$, and a sheaf $\cE$ of modules over $R$, we can define a subset $\musupp(\cE)\subset T^*M$ by interpreting Tamarkin's non-conic microsupport. This refines Kashiwara--Schapira's microsupport~\cite{KS}.

The second one is an introduction of curved sheaves and twisted sheaves. When we would like to relate sheaves with Floer theory, such notions are very natural, since Floer complex can be curved and allowed to be deformed by some bulk classes. 

\subsubsection{Algebraic asymptotic analysis}
We also discuss an application to asymptotic analysis. Let $f(\hbar)$ be a function which is asymptotically expandable toward $\hbar\to 0$ (say, from the real positive direction), namely it has an associated power series expansion. In this expansion, functions like $e^{-c/\hbar}$ ($c\in \bR_{>0}$) are ignored, since they have zero asymptotic expansion. However, sometimes, one can expand the function further into a formal power series of $e^{-c/\hbar}$, by seeing more precise asymptotics. The resulting expansion is called transseries, and by interpreting $e^{-c/\hbar}$ as $T^c$, the ring of transseries is a Novikov ring.

On the other hand, in \cite{WKBkuw}, the author constructed an object of $\Sh^{\bR_d}_{>0}(X\times \bR_t,\bK)$ from a second-order $\hbar$-differential equation over a Riemann surface, by using exact WKB analysis: From exact WKB analysis, one obtains a global version of Stokes matrices (called Voros matrices). Gluing by Voros matrices give an object of $\Sh^{\bR_d}_{>0}(X\times \bR_t,\bK)$.

The higher-order version of exact WKB analysis is not well-established yet, it has been expected that the resulting solutions and Stokes matrices are expressed by transseries. So, the natural place where we can find solutions would be sheaves over the Novikov ring. Then, to generalize the picture from \cite{WKBkuw} to the higher-order cases, we again need to relate $\Sh^{\bR_d}_{>0}(X\times \bR_t,\bK)$ and the category of sheaves over the Novikov ring, which is again the job of our main theorem in this paper.

\subsection*{Notation}
Let $\bK$ be a field. All the operations in this paper should be understood in the derived sense unless otherwise stated. For a $\bK$-algebra $A$, $\Mod(A)$ (resp. $\Mod^\heartsuit(A)$) is the derived $\infty$-category of right $A$-modules (resp. the abelian category of right $A$-modules).

\subsection*{Acknowledgment}
This work is supported by JSPS KAKENHI  Grant Numbers 22K13912, 23H01068, 20H01794, and 25K21660. I'd like to thank Yuichi Ike for related discussions. I also thank Tomohiro Asano, Takumi Arai, Bingyu Zhang, and anonymous referees for several helpful comments on earlier drafts.

\section{Novikov rings}
In this section, we define Novikov rings and explain some properties.

\subsection{Definition}
Let $\bR$ be the 1-dimensional Euclidean vector space. Let $\bG$ be a subgroup of $\bR$. Then $\bR_{\geq 0}\cap \bG$ has a semigroup structure with respect to the addition. We denote the corresponding polynomial ring by $\bK[\bR_{\geq 0}\cap \bG]$. We denote the indeterminate corresponding to $a\in\bR_{\geq 0}\cap \bG$ by $T^a$. Let $|\cdot|$ be the Euclidean norm of $\bR$. For $r\in \bR_{>0}$, we denote the ideal of $\bK[\bR_{\geq 0}\cap \bG]$ generated by $T^a$'s with $a>r$ by $\frakm(r)$. Obviously, $\frakm(r')\supset \frakm(r)$ if $r>r'$. Hence $\bK[\bR_{\geq 0}\cap \bG]/\frakm(r)$ forms a projective system. 
\begin{definition}
    The Novikov ring $\Lambda_0^{\bG}$ associated to $\bG$ is defined by
    \begin{equation}
\Lambda_0^{\bG}:=\lim_{\substack{\longleftarrow \\ r\rightarrow +\infty}}\bK[\bR_{\geq 0}\cap \bG]/\frakm(r).
\end{equation}
\end{definition}

\begin{example}[The universal Novikov ring]
    If $\bG=\bR$, the definition can be read as follows: Consider the semigroup of the non-negative real numbers $\bR_{\geq 0}$. We consider the polynomial ring $\bK[\bR_{\geq 0}]$. We denote the indeterminate corresponding to $a\in \bR_{\geq 0}$ by $T^a$. We set
\begin{equation}
\Lambda_0:=\Lambda_0^{\bR}\cong \lim_{\substack{\longleftarrow \\ a\rightarrow +\infty}}\bK[\bR_{\geq 0}]/T^a\bK[\bR_{\geq 0}].
\end{equation}
The ring is very useful in symplectic topology~\cite{FOOO} to control the energy/disk area. 
\end{example}

\begin{example}
    Similarly, if $\bG=\bZ$, we get the formal power series ring $\Lambda_0^{\bZ}\cong \bK[[T]]$.
\end{example}

\begin{example}
    Similarly, if $\bG=\bO:=\{0\}$, we get $\Lambda_0^\bO\cong \bK$.
\end{example}

We list up some properties of $\Lambda_0^\bG$.
\begin{lemma}
    For any $\bG$, the ring $\Lambda_0^\bG$ is 
    \begin{enumerate}
        \item an integral domain, and 
        \item a local ring.
    \end{enumerate}
\end{lemma}
\begin{proof}
    These are obvious. The maximal ideal $\frakm$ is given by the ideal generated by $\lc T^a\relmid a\in \bG\cap \bR_{>0} \rc$.
\end{proof}

\subsection{Modules over the Novikov ring}
Let $R$ be a Novikov ring i.e., $R=\Lambda_0^\bG$ for some $\bG$. We denote the abelian category of $R$-modules by $\Mod^\heartsuit(R)$. We also denote its derived category by $\Mod (R)$. We regard it as an $\infty$-category (or more concretely, a dg-category).

Let $X$ be a topological space. Let $R_X$ be the constant sheaf valued in $R$. We denote the abelian category $R_X$-modules by $\Sh^\heartsuit(X,R)$. We similarly consider the derived category of $\Sh^\heartsuit(X,R)$, and denote it by $\Sh(X,R)$, viewed as an $\infty$-category (or more concretely, a dg-category). Note that $\Sh(X,R)$ is equivalent to the $\infty$-category of sheaves valued in $\Mod(R)$ if $X$ is hypercomplete~\cite{ScholzeSixfunctor}.

In the case of $\bG\neq \bR$, to compare with sheaf theory, we introduce a noncommutative algebra as follows. We first introduce the following quiver: The set of vertices is $\bR/\bG$. For $[c], [c']\in \bR/\bG$, the set of arrows from $[c]$ to $[c']$ is denoted by
\begin{equation}
    \lc T_{[c]}^d \relmid d\in \bG\cap \bR_{\geq 0} \rc.
\end{equation}
The composition of the arrows follow the rule $T_{[c']}^{d'}T^d_{[c]}=T^{d+d'}_{[c]}$. Here we omit the target $[c']$ from the notation.

We denote the associated quiver algebra by $\bK[Q(\bR/\bG)]$. An element of this algebra is of the form
\begin{equation}
    \lb \sum_{d\in \bR_{\geq 0}}a_{[c],d}T_{[c]}^d\rb_{[c]\in \bR/\bG}
\end{equation}
where $a_{[c],d}\in \bK$ are zero except for finitely many $d$ for each $[c]$. For two elements, $\lb \sum_{d\in \bR_{\geq 0}}a_{[c],d}T_{[c]}^d\rb_{[c]\in \bR/\bG},     \lb \sum_{d\in \bR_{\geq 0}}b_{[c],d}T_{[c]}^d\rb_{[c]\in \bR/\bG} $, we define the product by
\begin{equation}
   \lb \sum_{d\in \bR_{\geq 0}}b_{[c],d}T_{[c]}^d\rb_{[c]\in \bR/\bG} \cdot\lb \sum_{d\in \bR_{\geq 0}}a_{[c],d}T_{[c]}^d \rb_{[c]\in \bR/\bG}= \lb\sum_{d''\in \bR_{\geq 0}}\lb \sum_{d+d'=d''}b_{[c+d], d'}a_{[c],d}\rb T_{[c]}^{d''}\rb_{[c]\in \bR/\bG}
\end{equation}
The sums are finite sums, so this is well-defined.

For $\ell>0$, we denote the ideal generated by the arrows represented by the positive numbers greater than $\ell$ by $\frakm(\ell)$. Explicitly,
\begin{equation}
    \frakm(\ell)=\lc \lb \sum_{d\in \bG\cap \bR_{\geq 0}}a_{[c],d}T^d_{[c]}\rb_{[c]\in \bR/\bG}\relmid a_{[c],d}=0 \text{ for } d\leq \ell\rc
\end{equation}
We take the completion of the algebra with $\frakm(\ell)$-adic topology and denote it by $L_0^{\bG}$. More explicitly, an element of $L_0^\bG$ is of the form
\begin{equation}
    \lb\sum_{d\in \bR_{\geq 0}}a_{[c],d}T_{[c]}^d\rb_{[c]\in \bR/\bG} 
\end{equation}
Here, for each $[c]$ and $L>0$, $a_{[c],d}\in \bK$ with $d\leq L$ are zero except for finitely many $d$ i.e., ``Novikov sum". The obvious multiplication is again well-defined.

By definition, the following is obvious:
\begin{lemma}\label{lem:persistentL}
As $\bK$-modules, $L_0^\bG\cong     \prod_{c\in \bR/\bG}\Lambda_0.$
\end{lemma}

We now would like to see several examples.
\begin{example}
The case of $\bR=\bG$. Since $\bR/\bG=\{*\}$, the quiver algebra is simply $\bK[Q(\bR/\bG)]=\bK[\bR_{\geq 0}]$. Hence the completion is $L_0^\bG=\Lambda_0$.
\end{example}

\begin{example}
Let $M$ be a persistence module, namely, a functor from the poset category $(\bR, \geq)$ to the category of vector spaces. For each $c\in \bR$, we denote the image under the functor by $M_c$. We suppose that there exists $L\in \bR$ such that $M_{c}=0$ for any $c<L$.
In the following, we will see that $\prod_{c\in \bR}M_c$ carries a $L_0^\bO$-module structure. 

For $c\leq c'$, we denote the structure morphism by $t_{c,c'}\colon M_c\rightarrow M_{c'}$.
We set $N_{-c}:=M_c$. Take an element $\prod_{c\in \bR} n_c=\prod_{c\leq -L}n_c\in \prod_{c\in \bR} N_c$. We also set $T_{[c]}^{d}=:f_{c, c+d}$.
The action is defined by
    \begin{equation}
        \lb         \prod_{c\in \bR} \sum_{c\leq c'}a_{c,c'}f_{c,c'}\rb\cdot \prod_{c\in \bR}n_c=\prod_{c\in \bR} n'_c
    \end{equation}
where 
\begin{equation}\label{eq:action}
    M_{-c}:=N_{c}\ni n_c'= \sum_{c\leq c'}a_{c,c'}t_{-c',-c}(n_{c'}).
\end{equation}
If $c'$ is sufficiently large, $n_{c'}=0$. Hence the sum (\ref{eq:action}) is a finite sum, hence well-defined.
\end{example}

The case of $\Lambda_0$ (i.e., $\bG=\bR$) is universal in the following sense:
\begin{lemma}
    We have an action of $\Lambda_0$ on $L_0^{ \bG}$. In particular, we have the forgetful functor
    \begin{equation}
        \frakf\colon \Mod(L_0^{\bG})\rightarrow \Mod(\Lambda_0).
    \end{equation}
\end{lemma}
\begin{proof}
    On $T_{[c]}^{d}$, the element $T^{d'}\in \Lambda_0$ acts on it by $T_{[c]}^{d}\mapsto T_{[c]}^{d+d'}$. This defines the desired action.
\end{proof}

\subsection{Real valuation}
We will later use the notion of real valuation. 
\begin{definition}
    Let $A$ be an integral domain with a map $v\colon A\rightarrow \bR_{\geq 0}\cup \lc +\infty\rc$. We say $(A, v)$ is a real valuation ring if
    \begin{enumerate}
    \item $v(0)=+\infty$,
        \item $v(a+b)\geq \min\lc v(a),v(b)\rc$,
        \item $v(ab)=v(a)+v(b)$
    \end{enumerate}
    for any $a,b\in A$.
\end{definition}

\begin{example}
  For the Novikov ring $\Lambda_0$, we set
    \begin{equation}
        v(a):=\min\lc c\in \bR_{\geq 0}\relmid a_c\neq 0\rc.
    \end{equation}
    where $a=\sum_{c\in \bR_{\geq 0}}a_cT^c$. This obviously gives a real valuation of $\Lambda_0$.
\end{example}

\section{Derived complete modules}
In this section, we recall some basic properties of derived complete modules. Our references are \cite{Kedlaya} and \cite{Stacks}.

\subsection{Derived completeness}
We first recall the definition of the completeness/derived completeness. 
\begin{definition}
Let $A$ be a ring and $I$ be a finitely generated ideal of $A$. Let $M$ be an $A$-module. The inverse limit $\displaystyle{\lim_{\substack{\longleftarrow\\ n\rightarrow \infty}}M/I^nM}$ is called the $I$-adic completion of $M$. 

We say $M$ is $I$-adically complete if the natural morphism 
\begin{equation}
    M\rightarrow \lim_{\substack{\longleftarrow\\ n\rightarrow \infty}}M/I^nM
\end{equation}
is an isomorphism. In other words, $M$ is complete with respect to $I$-adic topology.
\end{definition}

\begin{example}
    We consider the case of $\Lambda_0$ and $I$-adic completeness where $I=T\Lambda_0$.
    \begin{enumerate}
        \item $\Lambda_0$ itself is complete.
        \item $\bigoplus_\bN\Lambda_0$ is not complete. The completion is denoted by $\widehat{\bigoplus}_\bN\Lambda_0$. Concretely, it consists of a sequence $(x_i)_{i\in \bN}$ of $\Lambda_0$ satisfying $\lim_{i\rightarrow \infty}v(x_i)=\infty$ for the valuation $v$.
        \item $(\widehat{\bigoplus}_\bN\Lambda_0)\otimes_{\Lambda_0}(\widehat{\bigoplus}_\bN\Lambda_0)$ is not complete. Let us write the basis explicitly: $(\widehat{\bigoplus}_{i\in \bN}\Lambda_0e_i)\otimes_{\Lambda_0}(\widehat{\bigoplus}_{j\in \bN}\Lambda_0f_j)$. For example, in the completion, we have $\sum_{i=0}^\infty T^ie_i\otimes f_i$, but that is not in $(\widehat{\bigoplus}_\bN\Lambda_0)\otimes_{\Lambda_0}(\widehat{\bigoplus}_\bN\Lambda_0)$.
        
        In particular, the category of complete modules is not monoidal with respect to $\otimes_{\Lambda_0}$. 
        \item $\prod_{\bN}\Lambda_0$ is complete.
        \item $\Lambda_0[T^{-1}]$ is not complete, since $\Lambda_0[T^{-1}]/I^n\Lambda_0[T^{-1}]=0$ for any $n$.
        \item $\Lambda_0/\frakm$ is complete, since $(\Lambda_0/\frakm)/I^n(\Lambda_0/\frakm)=\Lambda_0/\frakm$.
    \end{enumerate}
\end{example}

It is known that the category of complete modules in general does not form an abelian category:
\begin{example}[{Adaptation of \cite[110.10]{Stacks}}]\label{ex:nonabel}
        We consider the following map
    \begin{equation}
        \varphi\colon \widehat{\bigoplus}_{n\in \bN}\Lambda_0\rightarrow \prod_{n\in \bN}\Lambda_0, (x_1,x_2,x_3,...)\mapsto (x_1, Tx_2, T^2x_3,...).
    \end{equation}
    in the full subcategory of the complete modules in $\Mod^\heartsuit(\Lambda_0).$
    We would like to check the homomorphism theorem. We first compute the coimage. The coimage is defined by the cokernel of the map $\ker(\varphi) \rightarrow \widehat{\bigoplus}_{n\in \bN}\Lambda_0$. Hence it is isomorphic to $\widehat{\bigoplus}_{n\in \bN}\Lambda_0$, since $\varphi$ is injective.
    
    On the other hand, the image $\image(\varphi)$ is defined by the kernel of the map $\prod_{n\in \bN}\Lambda_0\rightarrow \coker(\phi)$. The cokernel is defined by the completion of the cokernel $\coker^{\heartsuit}(\varphi)$ taken in $\Mod^\heartsuit(\Lambda_0)$. Since
    the element $(1,T,T^2,..)\in \prod_\bN\Lambda_0$ is not coming from $\varphi$, it defines a nontrivial element in $\coker^{\heartsuit}(\varphi)$. But, for any $n$, the class defined by $(1,T,T^2,..)$ modulo $\frakm^n$ is hit by $\varphi$. Hence $(1,T,T^2,..)$ is zero in the completion. Hence $(1,T,T^2,...)\in \image(\varphi)$. Hence the canonical morphism $\coim(\varphi)\rightarrow \image(\varphi)$ is not an isomorphism.
\end{example}
As we have seen, the notion of complete modules does not behave well homologically. For this reason, we use the notion of derived complete modules.
\begin{definition}
Let $A$ be a ring and $I$ is a finitely generated ideal of $A$. We say an object $M\in \Mod(A)$ is derived complete if $\Hom_{\Mod(A)}^\bullet(A[f^{-1}], M)\cong 0$ for any $f\in I$. 
\end{definition}
We have the following properties.
\begin{lemma}[e.g. {\cite{Kedlaya}}]
Let $A$ be a ring and $I$ is a finitely generated ideal of $A$.
\begin{enumerate}
\item Any complete module is derived complete.
\item Suppose $M\in \Mod^\heartsuit(A)$ is separated i.e., $\bigcap_n I^nM=0$ and derived complete. Then $M$ is complete.
\end{enumerate}
\end{lemma}

\begin{definition}
We denote the subcategory of derived complete modules with respect to $I=(T^1)$ of $\Mod(\Lambda_0)$ by $\Mod_\comp(\Lambda_0)$. We also set
\begin{equation}
    \Mod_\comp(L_0^{\bG}):=\frakf^{-1}(\Mod_\comp(\Lambda_0)).
\end{equation}
\end{definition}
\begin{lemma} \label{lem:presentabilityofL}
\begin{enumerate}
\item The inclusion $\Mod_{\comp}(L_0^{\bG})\subset \Mod(L_0^{\bG})$ admits a left adjoint. We call it the completion and denote it by $M\mapsto \widehat{M}$. Explicitly, it is given by 
\begin{equation}\label{eq:completionformula}
    \widehat M:=\lim_{\substack{\longleftarrow \\ r\rightarrow +\infty}}L_0^\bG/\frakm(r)\otimes_{L_0^\bG}M
\end{equation}
\item $\Mod_{\comp}(L_0^{\bG})$ is a stable and presentable category. We denote the coproduct in $\Mod_{\comp}(L_0^{\bG})$ by $\widehat{\bigoplus}$.
\end{enumerate}
\end{lemma}
\begin{proof}
Due to some standard facts (see e.g. \cite[Proposition 6.6.2]{Kedlaya} or \cite[Lemma 15.93.18]{Stacks}), the derived complete objects is characterized by the right orthogonal to the homotopy limit of the sequence
\begin{equation}
    N_c:=\cdots \xrightarrow{T^c}\Lambda_0\xrightarrow{T^c}\Lambda_0\xrightarrow{T^c}\Lambda_0.
\end{equation}
for some $c>0$ (and any $c>0$). We also set
\begin{equation}
        N_c^\bG:=\cdots \xrightarrow{T^c}L^\bG_0\xrightarrow{T^c}L^\bG_0\xrightarrow{T^c}L^\bG_0.
\end{equation}
for $c\in \bG\cap \bR_{>0}$. It is easy to see that the vanishing of $\Hom_{\Mod(L_0^\bG)}(N_c^\bG, M)$ is equivalent to the vanishing of $\Hom_{\Mod(\Lambda_0)}(N_c, \frakf(M))$. Hence $\Mod_\comp(L_0^\bG)$ is the full subcategory of the $S$-local objects where $S=\lc N_c^\bG \rc_{c\in \bG\cap \bR_{>0}}$. Since $\Mod(L_0^\bG)$ is compactly generated, \cite[5.5.4.15]{HTT} implies that $\Mod_\comp(L_0^\bG)$ is presentable and the inclusion $\Mod_\comp(L_0^\bG)\hookrightarrow \Mod(L_0^\bG)$ admits a left adjoint. Now it is easy to see the left adjoint is given by the formula (\ref{eq:completionformula}), and the coproduct is given by
\begin{equation}
    \widehat\bigoplus_{i\in I}\cE_i:=\widehat{\bigoplus_{i\in I}\cE_i}.
\end{equation}
\end{proof}

\begin{example}
    \begin{enumerate}
    \item From the above examples, $\Lambda_0, \prod_\bN\Lambda_0, \Lambda_0/\frakm$ are complete, hence derived complete.
    \item The coproduct $\bigoplus_\bN\Lambda_0$ is separated and not complete, hence not derived complete.
    \item $\Lambda_0[T^{-1}]$ is not complete, not separated. Since $\Hom_{\Lambda_0}(\Lambda_0[T^{-1}],\Lambda_0[T^{-1}])$ is not zero, the module $\Lambda_0[T^{-1}]$ is not derived complete.
    \item The cokernel of the map
    \begin{equation}
        \widehat{\bigoplus}_{n\in \bN}\Lambda_0\rightarrow \prod_{n\in \bN}\Lambda_0, (x_1,x_2,x_3,...)\mapsto (x_1, Tx_2, T^2x_3,...)
    \end{equation}
    in Example~\ref{ex:nonabel} taken in $\Mod^\heartsuit(\Lambda_0)$ is not complete. But it gives an exact triangle in $\Mod(\Lambda_0)$, hence derived complete.
    \end{enumerate}
\end{example}

\section{Almost modules}
In the later comparison, we have some discrepancy between sheaf and Novivkov ring which can be ignored by using almost mathematics. In this section, we recall some basic constructions. We refer to \cite{GabberRamero} for general ideas of almost mathematics.
\subsection{Almost isomorphism in $\Mod(R)$}
We have the usual Novikov ring $\Lambda_0$ and its maximal ideal $\frakm$. Note that $\frakm$ is flat.
\begin{definition}
\begin{enumerate}
    \item For $M\in \Mod(\Lambda_0)$, we say $M$ is almost zero if $M\otimes_{\Lambda_0} \frakm=0$.
    \item     Let $f\colon M\rightarrow N$ be a morphism of $\Lambda_0$-modules. We say $f$ is an almost isomorphism if $\ker(f)$ and $\coker(f)$ are almost zero modules.
\end{enumerate}
\end{definition}
We note that the full subcategory $\Sigma$ of $\Mod(\Lambda_0)$ consisting of the almost zero modules is a thick subcategory. 
We take the quotient $\Mod(\Lambda_0^\almost):=\Mod(\Lambda_0)/\Sigma$. We denote the quotient functor by 
\begin{equation}
    \fraka\colon \Mod(\Lambda_0)\rightarrow \Mod(\Lambda_0^\almost).
\end{equation}
There exists a right adjoint~\cite{GabberRamero} denoted by
\begin{equation}
    (-)_*:=\Hom_{\Mod(\Lambda_0^\almost)}(\Lambda_0,-)\colon \Mod(\Lambda_0^\almost)\rightarrow \Mod(\Lambda_0); M\mapsto M_*.
\end{equation}

We next consider $R:=L_0^\bG$ for some $\bG$. We have the forgetful functor $\Mod(L_0^\bG)\rightarrow \Mod(\Lambda_0)$. We set 
\begin{equation}
    \Sigma_{R}:=\frakf^{-1}(\Sigma).
\end{equation}
 We set
\begin{equation}
    \Mod(R^\almost):=\Mod(R)/\Sigma_R.
\end{equation}
We denote the quotient functor by
\begin{equation}
    \fraka\colon \Mod(R)\rightarrow \Mod(R^\almost).
\end{equation}
We say $M, N\in \Mod(R)$ are almost isomorphic if $\fraka(M)\cong \fraka(N)$.

\subsection{Almost isomorphisms between $\cC$-modules}
Let $\cC$ be a stable $\Lambda_0$-linear category.

\begin{definition}
    \begin{enumerate}
        \item Let $\Mod(\cC,\Lambda_0)$ be the category of $\cC$-modules over $\Lambda_0$. An almost zero module $M$ is a module such that $\frakm\otimes M\cong 0$. In other words, $M(c)$ is almost zero for any $c\in \cC$.
        \item We denote the category of almost modules by $\Mod(\cC, \Lambda_0^\almost)$ which is the quotient by the almost zero modules.
         \item Two objects $M, N$ in $\Mod(\cC, \Lambda_0)$ are said to be almost isomorphic if they are isomorphic in $\Mod(\cC, \Lambda_0^\almost)$.
        \item  We denote the $\Lambda_0$-linear Yoneda embedding by $\cY\colon \cC\rightarrow \Mod(\cC, \Lambda_0)$. We say $\cE, \cF\in \cC$ are almost isomorphic if  $\cY(\cE)$ and $\cY(\cF)$ are almost isomorphic. In this case, we denote it by $\cE\cong_\almost\cF$.
    \end{enumerate}
\end{definition}

\begin{lemma}
    Let $f,f'\colon \cE\rightarrow \cF$ be morphisms such that $f-f'$ is almost zero. Then $\Cone(f)\cong_\almost\Cone(f')$.
\end{lemma}
\begin{proof}
    Recall that $\frakm$ is flat. Then 
    \begin{equation}
       \frakm \otimes_{\Lambda_0}  \cY(\Cone(f))=\Cone(\frakm \otimes_{\Lambda_0} \cY(f)\colon \frakm \otimes_{\Lambda_0} \cY(\cE)\to \frakm \otimes_{\Lambda_0} \cY(\cF)).
    \end{equation}
    Since $f-f'$ is almost zero, $\frakm \otimes_{\Lambda_0} \cY(f)=\frakm \otimes_{\Lambda_0} \cY(f')$. Hence we have an isomorphism
\begin{equation}\label{eq:almost_isom_cone}
    \frakm \otimes_{\Lambda_0} \cY(\Cone(f'))\cong \frakm \otimes_{\Lambda_0}\cY(\Cone(f)).
\end{equation}
Since $\frakm$ is almost isomorphic to $\Lambda_0$, we have
\begin{equation}
\cY(\Cone(f'))\cong_\almost\frakm \otimes_{\Lambda_0} \cY(\Cone(f')) \cong \frakm \otimes_{\Lambda_0} \cY(\Cone(f)) \cong_\almost\cY(\Cone(f)).
\end{equation}
where the middle equality uses (\ref{eq:almost_isom_cone}).
\end{proof}

For general $\bG$, we slightly modify the setup:
\begin{definition}
Let $\bG$ be a subgroup of $\bR$.
    A category over $L_0^\bG$ is a tuple
    \begin{enumerate}
        \item A cocomplete category $\cC$, and
        \item a group homomorphism $T_\bullet\colon \bR/\bG\rightarrow \Aut(\cC)$.
    \end{enumerate}
    with a homomorphism
    \begin{equation}
        L_0^\bG\rightarrow\End(\bigoplus_{c\in \bR/\bG}T_c).
    \end{equation}
\end{definition}

Let $\cC$ be a category over $L_0^\bG$. Then we have a functor
\begin{equation}
    \cY\colon \cC\rightarrow \Mod(\cC, L_0^\bG):=\Fun(\cC, \Mod(L_0^\bG)); \cE\mapsto \Hom(\bigoplus_{c\in\bR/\bG}T_c\cE,-)
\end{equation}
where $\Fun$ is the functor category.

\begin{definition}
    \begin{enumerate}
        \item An almost zero module $M\in \Mod(\cC, L_0^\bG)$ is a module such that $M(c)\in \Sigma_{L_0^\bG}$ for any $c\in \cC$. 
        \item We denote the category of almost modules by $\Mod(\cC,{L_0^\bG}^\almost)$ which is the quotient by the almost zero modules.
        \item Two objects $M, N$ in $\Mod(\cC, {L_0^\bG})$ are said to be almost isomorphic if they are isomorphic in $\Mod(\cC, {L_0^\bG}^\almost)$.
        \item We say $\cE, \cF\in \cC$ are almost isomorphic if  $\cY(\cE)$ and $\cY(\cF)$ are almost isomorphic. In this case, we denote it by $\cE\cong_\almost\cF$.
    \end{enumerate}
\end{definition}
Similarly, we can prove that the extensions by almost same morphisms are almost isomorphic.

\subsection{Almost equivalence}
\begin{definition}
Let $\cC_1, \cC_2$ be categories defined over $\Lambda_0$.
    Let $F\colon \cC_1\rightarrow \cC_2$ be a $\Lambda_0$-linear functor. We say $F$ is an almost equivalence if it satisfies the following:
    \begin{enumerate}
        \item For any $c, c'\in \cC_1$, the induced morphism
        \begin{equation}
            \Hom_{\cC_1}(c, c')\rightarrow  \Hom_{\cC_2}(F(c), F(c'))
        \end{equation}
        is an almost isomorphism.
        \item For any $c' \in \cC_2$, there exists $c\in \cC_1$ such that $F(c)$ is almost isomorphic to $c'$.
    \end{enumerate}
\end{definition}

In the following, we give a little generalization of the above notion: 
\begin{definition}
    Let $\cC_1, \cC_2$ be cocomplete categories over $L_0^\bG$. 
    \begin{enumerate}
        \item     A morphism $F\colon \cC_1\rightarrow \cC_2$ is a functor from $\cC_1$ to $\cC_2$ commuting with $T_\bullet$.
        \item We say a morphism $F\colon \cC_1\rightarrow \cC_2$ is almost fully faithful (, or almost embedding, $\cC_1\hookrightarrow_\almost \cC_2$) if the following holds: For any $\alpha, \alpha'\in \cC_1$, the induced morphism
        \begin{equation}
            \Hom_{\cC_1}(\bigoplus_{c\in \bR/\bG}T_c\alpha, \alpha')\rightarrow  \Hom_{\cC_2}(\bigoplus_{c\in \bR/\bG}T_cF(\alpha), F(\alpha'))
        \end{equation}
        is an almost isomorphism in $\Mod(L_0^\bG)$.
        \item We say a morphism $F\colon \cC_1\rightarrow \cC_2$ is almost essentially surjective if the following holds: For any $\alpha' \in \cC_2$, there exists $\alpha\in \cC_1$ such that $F(\alpha)$ is almost isomorphic to $c'$.
        \item We say $\cC_1$ and $\cC_2$ are almost equivalent (, or $\cC_1\cong_\almost \cC_2$) if there exists a functor $f$ from $\cC_1$ to $\cC_2$ such that $f$ is almost fully faithful and almost essentially surjective.
    \end{enumerate}
\end{definition}

We will deal with the following two examples:
\begin{example}
    We consider the category $\Mod(L_0^\bG)$. Since an object of $\Mod(L_0^\bG)$ carries an $\bR/\bG$-grading, we can shift it. We denote the resulting shift functor by $T_c$. Then, for any $M, N\in \Mod(L_0^\bG)$, we have
    \begin{equation}
        \Hom_{\Mod(L_0^\bG)}(\bigoplus_{c\in \bR/\bG} T_cM, N),
    \end{equation}
    which is a right $L_0^\bG$-module. 
\end{example}

\begin{example}
The category $\mu^\bG(T^*X)$ will be introduced in the next section. We have shift operations $T_c$ parametrized by $c\in \bR/\bG$. Then, for any $\cE, \cF\in \mu^\bG(T^*X)$, we have
    \begin{equation}
        \Hom_{ \mu^\bG(T^*X)}(\bigoplus_{c\in \bR/\bG} T_c\cE, \cF),
    \end{equation}
    which is a right $L_0^\bG$-module. 
\end{example}

\section{Equivariant sheaves and the Novikov ring}
In this section and the next section, we relate Tamarkin categories with Novikov rings precisely.
\subsection{Basics}
Let $\bR_t$ be the $1$-dimensional real vector space with the standard coordinate $t$. 

We consider the addition action of a subgroup $\bG\subset \bR$ on $\bR_t$ as a discrete group action. Then we can consider the derived $\infty$-category of equivariant $\bK$-module sheaves $\Sh^{\bG}(\bR_t,\bK)$. The notion of microsupport $\SS$ is defined for each object of $\Sh(\bR_t,\bK)$. For each object $\Sh^{\bG}(\bR_t,\bK)$, we define its microsupport by that of the image under the forgetful functor $\Sh^{\bG}(\bR_t,\bK)\to \Sh(\bR_t,\bK)$ forgetting the equivariancy.

We denote the subcategory spanned by the object whose microsupport contained in $\bR\times \bR_{\leq 0}\subset \bR\times \bR\cong T^*\bR_t$ by $\Sh^{\bG}_{\bR_{\leq 0}}(\bR_t,\bK)$. We set
\begin{equation}
   \mu^{\bG}(*):= \Sh^{\bG}_{\bR_{>0}}(\bR_t,\bK):=\Sh^{\bG}(\bR_t,\bK)/\Sh^{\bG}_{\bR_{\leq 0}}(\bR_t,\bK).
\end{equation}
By Tamarkin's cutoff functor (see \cite{WKBkuw}), the category $\mu^\bG(*)$ has a canonical fully faithful embedding $\mu^\bG(*)\hookrightarrow \Sh^\bG(\bR, \bK)$. Most of the hom-computations below are done through this embedding.

We have the following:
\begin{lemma}[\cite{WKBkuw, hRH}]\label{lem:unitend}
\begin{enumerate}
    \item $\mu^{\bG}(*)$ has a monoidal structure $\star$ defined by the convolution product.
    \item  We equip the sheaf $1_\mu:=\bigoplus_{c\in \bG}\bK_{t\geq c}$ with an obvious equivariant structure. Then it defines an object of $\mu^{\bG}(*)$ and is a monoidal unit.
    \item We have  $H^0\End_{\mu^{\bG}(*)}(1_\mu) \cong \Lambda_0^{\bG}$. As a corollary of 2 and 3, $\mu^{\bG}(*)$ is enriched over $\Lambda_0^{\bG}$.
    \item The forgetful functor $\mu^\bG(*)\rightarrow \mu^\bO(*)$ forgetting the equivariancy admits left adjoint given by $\cE\mapsto \bigoplus_{c\in \bG}\bK_{t\geq c}\star \cE$ with the obvious equivariant structure on the image.
\end{enumerate}
\end{lemma}

We strengthen the result a little more.
\begin{lemma}\label{lem:almostisom1}
We have an almost isomorphism of almost $\Lambda_0^\bG$-modules
\begin{equation}
    \End_{\mu^{\bG}(*)}(1_\mu)\cong_\almost\Lambda_0^{\bG}.
\end{equation}
\end{lemma}
More precisely, the higher cohomologies of $\End_{\mu^{\bG}(*)}(1_\mu)$ are almost zero, but not zero.
\begin{proof}
For the case when $\bG=\{0\}$ and $\bG\cong \bZ$, there are no higher cohomologies. In the following, we only prove the case when $\bG=\bR$. Other cases (i.e., dense subgroups of $\bR$) can be proved similarly.

In the rest of this proof, $\Hom$ means dg-Hom space in $\Sh(\bR,\bK)$. 
    We first consider the following exact triangle:
    \begin{equation}\label{5.2_extriangle1}
        \Hom(\bK_{t\geq 0}, \bigoplus_{c\in \bR} \bK_{t\geq  c})\rightarrow \Hom(\bK_{\bR}, \bigoplus_{c\in \bR} \bK_{t\geq  c})\rightarrow \Hom(\bK_{t<0}, \bigoplus_{c\in \bR} \bK_{t\geq c})\rightarrow.
    \end{equation}
    Since the $i$-th cohomologies vanish for $i>1$ due to \cite[Proposition 3.2.2]{KS}, we will only take care of $H^1\Hom(\bK_{\bR}, \bigoplus_{c\in \bR} \bK_{t\geq  c})$. 

    Now we consider the exact triangle:
    \begin{equation}
        \Hom(\bK_{\bR}, \bigoplus_{c\in \bR} \bK_{t< c})\rightarrow \Hom(\bK_{\bR}, \bigoplus_{c\in \bR} \bK_{\bR})\rightarrow \Hom(\bK_{\bR}, \bigoplus_{c\in \bR} \bK_{t\geq c})\rightarrow 
    \end{equation}
    Note that 
    \begin{equation}
        \Hom(\bK_{\bR}, \bigoplus_{c\in \bR} \bK_{\bR})\simeq \bigoplus_{c\in \bR} \bK
    \end{equation}
    since $\bK_\bR$ is compact in the full subcategory of constant sheaves on $\bR$.
    Hence we have $H^1\Hom(\bK_{\bR}, \bigoplus_{c\in \bR} \bK_{t\geq c})\simeq 0$. This together with (\ref{5.2_extriangle1}) implies that 
    \begin{equation}\label{eq:coker}
        H^1\Hom(\bK_{t\geq 0}, \bigoplus_{c\in \bR} \bK_{t\geq  c})\simeq \coker(H^0\Hom(\bK_{\bR}, \bigoplus_{c\in \bR} \bK_{t\geq  c})\rightarrow H^0\Hom(\bK_{t<0}, \bigoplus_{c\in \bR} \bK_{t\geq  c})).
    \end{equation}
    Note that the right hand side is isomorphic to 
    \begin{equation}
            \coker(H^0\Hom(\bK_{-a<t}, \bigoplus_{c\in \bR} \bK_{t\geq  c})\rightarrow H^0\Hom(\bK_{-a<t<0}, \bigoplus_{c\in \bR} \bK_{t\geq  c}))
    \end{equation}
 for any $a>0$ by the non-characteristic deformation lemma. This implies that the action of $T^a$ on (\ref{eq:coker}) is zero for any $a>0$. Hence it is almost zero. This completes the proof.
\end{proof}
We further have the following:
\begin{lemma}\label{lem:almostisom2}
We have an almost isomorphism of almost $L_0^\bG$-modules
\begin{equation}
    \End_{\mu^{\bG}(*)}( \bigoplus_{c\in \bR/\bG}T_c1_\mu)\cong_\almost L_0^{\bG}.
\end{equation}
\end{lemma}
\begin{proof}
By using the preceding lemma, we have a sequence of almost isomorphisms
\begin{equation}
    \begin{split}
            \End_{\mu^{\bG}(*)}( \bigoplus_{c\in \bR/\bG}T_c1_\mu)\cong_\almost \prod_{c\in \bR/\bG}\Hom_{\mu^{\bO}(*)}(\bK_{t\geq 0}, \bigoplus_{c\in \bR}\bK_{t\geq c})\cong_\almost \prod_{c\in \bR/\bG}\Lambda_0^\bG\cong_\almost L_0^\bG
    \end{split}
\end{equation}
By Lemma~\ref{lem:persistentL}, we get an isomorphism of the underlying $\bK$-modules. One can see that this isomorphism preserves the algebra structure.
\end{proof}

\subsection{Derived completeness}
\begin{lemma}\label{lem:derivedcompleteness}
For any $\cE, \cF\in \mu^\bG(*)$, we have $\Hom_{\mu^\bG(*)}(\bigoplus_{c\in \bR/\bG}T_c\cE,\cF)\in \Mod_\comp(L_0^\bG)$.
\end{lemma}
\begin{proof}
We first show the space $\Hom_{\mu^\bR(*)}(\cE, \cF)$ for $\cE, \cF$ is derived complete. 
In other words, it is enough to show the homotopy limit of the sequence
\begin{equation}
    \cdots\xrightarrow{T}\Hom(\cE, \cF)\xrightarrow{T}\Hom(\cE, \cF)\xrightarrow{T}\Hom(\cE, \cF).
\end{equation}
is zero~\cite[Cor 4.2.8]{DAGXII}. For the notation, we denote it by 
\begin{equation}
    \lim_{\substack{\longleftarrow\\ i\rightarrow \infty}}\Hom(\cE, \cF).
\end{equation}
By using the internal hom defined in \cite{Tam, WKBkuw}, we have
\begin{equation}
\begin{split}
     \lim_{\substack{\longleftarrow\\ i\rightarrow \infty}}\Hom_{\mu^\bR(*)}(\cE, \cF)&=\Hom_{\mu^\bR(*)}(\bigoplus_{c\in \bR}\bK_{t\geq c}, \lim_{\substack{\longleftarrow\\ i\rightarrow \infty}}\cHom^{\star_\bR}(\cE,\cF))\\
     &=\Hom_{\mu^\bO(*)}(\bK_{t\geq 0}, \lim_{\substack{\longleftarrow\\ i\rightarrow \infty}}\cHom^{\star_\bR}(\cE,\cF))\\
\end{split}
\end{equation}
where we use Lemma~\ref{lem:unitend}(4) for the second equality.

We set $\cHom^{\star_\bG}(\cE, \cF)=:\cG$. Here $\displaystyle{\lim_{\substack{\longleftarrow\\ i\rightarrow \infty}}\cG}$ is the homotopy limit of the sequence
\begin{equation}
    \cdots\xrightarrow{T} \cG\xrightarrow{T} \cG\xrightarrow{T} \cG.
\end{equation}
We then have
\begin{equation}
\begin{split}
    \Gamma(\bK_{\lc t\geq 0\rc},\lim_{\substack{\longleftarrow\\ i\rightarrow \infty}}\cG)
    &\cong \lim_{\substack{\longleftarrow\\ i\rightarrow \infty}}\Gamma(\bK_{\lc t\geq 0\rc }, \cG)\\
    &\cong \lim_{\substack{\longleftarrow\\ c\rightarrow \infty}}\Gamma(\bK_{\lc t\geq c\rc }, \cG)\\
    &\cong \Gamma(\lim_{\substack{\longrightarrow\\ c\rightarrow \infty}}\bK_{\lc t\geq c\rc }, \cG)\cong 0.
\end{split}
\end{equation}
Here, in the second equality, we use the following commutative diagram to transfer the $T$-action from $\cG$ to $\bK_{t\geq c}$:
\begin{equation}
    \xymatrix{
    \Gamma(\bK_{\lc t\geq 0\rc }, \cG)\ar[r]^{T\circ(-)}\ar[d]^\cong & \Gamma(\bK_{\lc t\geq 0\rc }, \cG)\ar[d]^{\cong} \\
    \Gamma(\bK_{\lc t\geq n\rc },\cG)\ar[r]&\Gamma(\bK_{\lc t\geq n+1\rc },\cG)
    }
\end{equation}
where the lower horizontal arrow is induced by the canonical morphism $\bK_{t\geq n}\to \bK_{t\geq n+1}$, and vertical arrows are induced by translations. This completes the proof.

For general $\bG$, we can equip $\bigoplus_{c\in \bR/\bG}T_c\cE$ with the $\bG$-equivariance  (which is the left adjoint of the forgetful functor $\mu^\bR(*)\to \mu^\bG(*)$. See Lemma ~\ref{lem:unitend} (4) for a related assertion). This structure induces $\Lambda_0$-action on $\Hom_{\mu^\bG(*)}(\bigoplus_{c\in \bR/\bG}T_c\cE,\cF)$, which is nothing but the structure of $\frakf(\Hom_{\mu^\bG(*)}(\bigoplus_{c\in \bR/\bG}T_c\cE, \cF))$. Hence the derived completeness of $\Hom_{\mu^\bG(*)}(\bigoplus_{c\in \bR/\bG}T_c\cE,\cF)$ is again the vanishing of 
the homotopy limit of the sequence
\begin{equation}
    \cdots\xrightarrow{T}\Hom(\bigoplus_{c\in \bR/\bG}T_c\cE, \cF)\xrightarrow{T}\Hom(\bigoplus_{c\in \bR/\bG}T_c\cE, \cF)\xrightarrow{T}\Hom(\bigoplus_{c\in \bR/\bG}T_c\cE, \cF),
\end{equation}
which can be shown by the same argument as above.
\end{proof}

\subsection{Morita functor}
By Lemma~\ref{lem:derivedcompleteness}, we have the Morita functor
\begin{equation}
   \frakA \colon \mu^\bG(*)\rightarrow \Mod_{\comp}(L_0^{\bG}); \cE\mapsto \Hom_{\mu^\bG(*)}(\bigoplus_{c\in \bR/\bG}T_{c}1_\mu, \cE).
\end{equation}

\begin{theorem}\label{thm:main}
    The functor $\frakA$ is an almost equivalence.
\end{theorem}

\begin{proof}[Proof of Theorem~\ref{thm:main}]

\subsection*{Step 1}
We first show the conservativity of the functor: If an object $\cE\in \mu^\bG(*)$ satisfies $\Hom_{\mu^\bG(*)}(\bigoplus_{c\in \bR/\bG}T_c1_\mu, \cE)=0$, then $\cE=0$. Suppose $\cE$ satisfies $\Hom_{\mu^\bG(*)}(\bigoplus_{c\in \bR/\bG}T_c 1_\mu, \cE)=0$. Since $\Cone(T^{c'}\colon 1_\mu\to 1_\mu)=\bigoplus_{c\in \bG} \bK_{c\leq t<c+c'}$ for any $c'>0$, we have
\begin{equation}
\begin{split}
        0&=\Hom_{\mu^\bG(*)}(\bigoplus_{c''\in \bR/\bG}T_{c''} \bigoplus_{c\in \bG} \bK_{c\leq t<c+c'}, \cE)\\ &\cong \Hom_{\Sh(\bR_t)}(\bigoplus_{c\in \bR/\bG}T_{\widetilde c} \bK_{c\leq t<c+c'}, \cE)\\
        &\cong \prod_{c\in \bR/\bG}\Hom_{\Sh(\bR_t)}(\bK_{c+\widetilde c\leq t<c+\widetilde c+c'}, \cE)
\end{split}
\end{equation}
where ${\widetilde c}$ is a representative of $c\in \bR/\bG$. Hence the microstalk of $\cE$ at any positive codirections vanish. Hence $\cE$ has non-positive microsupport. This implies $\cE=0$, since $\cE$ is an object of (equivariant) Tamarkin category.

\subsection*{Step 2}
We next prove the following:
\begin{claim}
    \begin{equation}
    \Hom(\bigoplus_{c\in \bR/\bG}T_c 1_\mu, \bigoplus_{i\in I}T_{d_i}1_\mu)\cong_\almost \widehat{\bigoplus}_{i\in I}\Hom(\bigoplus_{c\in \bR/\bG}T_c 1_\mu, T_{d_i}1_\mu)
\end{equation}
in $\Mod_{\comp}(L_0^{\bG})$.
\end{claim}
\begin{proof}[Proof of claim]
We consider the case of $\bG=\bR$. Other cases are similar.

We first prove 
    \begin{equation}\label{eq:firstequality}
    \Hom(T_c 1_\mu, \bigoplus_{i\in I}T_{d_i}1_\mu)\cong \widehat{\bigoplus}_{i\in I}\Hom(T_c 1_\mu, T_{d_i}1_\mu).
\end{equation}
We only consider the case when $c=0$, since other cases are similar. We first replace the left hand side with
\begin{equation}
    \Hom(1_\mu, \bigoplus_{i\in I}1_\mu)\cong \Hom(\bK_{t\geq 0}, \bigoplus_{I}\bigoplus_{c\in \bR}\bK_{t\geq c}).
\end{equation}
Then we have
\begin{equation}
\Hom(\bK_{t\geq 0}, \bigoplus_{I}\bigoplus_{c\in \bR}\bK_{t\geq c})\rightarrow
 \Hom(\bK_\bR, \bigoplus_{I}\bigoplus_{c\in \bR}\bK_{t\geq c})\rightarrow \Hom(\bK_{(-\infty, 0)}, \bigoplus_{I}\bigoplus_{c\in \bR}\bK_{t\geq c})
    \rightarrow.
\end{equation}
As in the proof of Lemma~\ref{lem:almostisom1}, we can see that $H^1( \Hom(\bK_\bR, \bigoplus_{I}\bigoplus_{c\in \bR}\bK_{t\geq c}))\simeq 0$. Then one can go to the cohomology exact sequence as
\begin{equation}
0\rightarrow 
H^0\Hom(\bK_{t\geq 0}, \bigoplus_{I}\bigoplus_{c\in \bR}\bK_{t\geq c})
    \rightarrow \widehat{\bigoplus}_{c\in (-\infty,+\infty),I}\bK\rightarrow \widehat{\bigoplus}_{c\in (-\infty,0),I}\bK \rightarrow H^1\Hom(\bK_{t\geq 0}, \bigoplus_{I}\bigoplus_{c\in \bR}\bK_{t\geq c})\rightarrow 0.
\end{equation}
Here $\widehat{\bigoplus}_{c\in (-\infty,+\infty),I}\bK$ is the subspace of $\prod_{c\in (-\infty,+\infty),I}\bK$ spanned by the elements of the form $\prod_{c\in (-\infty, +\infty), i\in I}a_{c,i}$ ($a_{c,i}\in \bK$) satisfying 
\begin{equation}
    \#\lc (c,i)\in (-\infty, N)\times I\relmid a_{c,i}\neq 0 \rc<+\infty
\end{equation}
for any $N\in \bR$. Similarly, $\widehat{\bigoplus}_{c\in (-\infty,0),I}\bK$ is the subspace of $\prod_{c\in (-\infty,0),I}\bK$ spanned by the elements of the form $\prod_{c\in (-\infty, 0), i\in I}a_{c,i}$ ($a_{c,i}\in \bK$) satisfying 
\begin{equation}
    \#\lc (c,i)\in (-\infty, N)\times I\relmid a_{c,i}\neq 0 \rc<+\infty
\end{equation}
for any $N\in \bR_{<0}$. This computation is done in exactly same manner as the proof of Lemma~\ref{lem:unitend}.3, where its proof can be found in \cite{WKBkuw}.

Hence we conclude that
\begin{equation}
    H^i\Hom(\bK_{t\geq 0}, \bigoplus_{I}\bigoplus_{c\in \bR}\bK_{t\geq c})\cong \begin{cases}
        & \Lambda_0\text{ if $i=0$}\\
        & \coker\lb \bigoplus_{(-\infty,0),I} \bK \rightarrow \widehat{\bigoplus}_{(-\infty,0),I} \bK\rb\text{ if $i=1$}\\
        &0\text{ otherwise}.
    \end{cases}
\end{equation}
where the morphism in the second line is a natural one. Then the almostization kills the degree one morphisms.

By (\ref{eq:firstequality}), we have
\begin{equation}
    \Hom(1_\mu, \bigoplus_{i\in I}T_{d_i}1_\mu)\cong_\almost \widehat \bigoplus_{i\in I}\Hom(T_c 1_\mu, T_{d_i}1_\mu).
\end{equation}
This completes the proof.
\end{proof}

\subsection*{Step 3}
With Step 1, the argument of \cite[Proposition 5.4.5]{GaitsgoryRosenblyum} immediately shows that $\bigoplus_{c\in \bR/\bG}T_c1_\mu$ generates the whole category by colimits. Since $\frakA(\bigoplus_{c\in \bR/\bG}T_c1_\mu)\cong_\almost L_0^\bG$ by Lemma~\ref{lem:almostisom2} and 
\begin{equation}
\End(\bigoplus_{c\in \bR/\bG}T_c1_\mu)\cong_\almost L_0^\bG\cong \End(L_0^\bG),
\end{equation}
we conclude that the functor $\frakA$ is almost fully faithful by using Step 2.

Also, since $L_0^\bG$ is a generator of $\Mod_{\comp}(L_0^\bG)$ and $\frakA$ is cocontinuous, the almost essential surjectivity follows. This completes the proof.
\end{proof}

\begin{proposition}[The inverse functor]\label{prop:inverse}
    The functor $\frakB\colon \Mod_\comp(L_0^\bG)\to \mu^\bG(*); M\mapsto M\otimes_{L_0^\bG}\lb \bigoplus_{c\in \bR/\bG} T_c1_\mu\rb$ is a left adjoint of $\frakA$ and an almost equivalence. Moreover $\frakA\circ \frakB(M)\cong_{\almost}M, \frakB\circ \frakA(\cE)\cong_\almost \cE$ for any $M,\cE$.
\end{proposition}
\begin{proof}
    For $M, N\in \Mod_\comp(L_0^\bG)$, we have
    \begin{equation}
        \Hom(M, \frakA(\cE))=\Hom(M, \Hom(\bigoplus_{c\in \bR/\bG}T_c1_\mu, \cE))\cong \Hom(M\otimes_{L_0^\bG}\lb \bigoplus_{c\in \bR/\bG} T_c1_\mu\rb, \cE).
    \end{equation}
    Hence $\frakB$ is a left adjoint of $\frakA$.
    
    The assertion $\frakA\circ \frakB(M)\cong_\almost M$ follows from the proof of the almost essential surjectivity of $\frakA$. We then have
    \begin{equation}
        \Hom(\frakB(M), \frakB(N))\cong \Hom(M, \frakA\circ \frakB(N))\cong_\almost \Hom(M,N),
    \end{equation}
    which shows the almost fully faithfulness of $\frakB$. Since $\frakA$ is almost fully faithful, we have
    \begin{equation}
        \Hom(\frakB\circ \frakA(\cE), \cF)\cong \Hom(\frakA(\cE), \frakA(\cF))\cong_{\almost} \Hom(\cE,\cF)
    \end{equation}
    This shows $\frakB\circ \frakA(\cE)\cong_\almost \cE$. This also show the almost essential surjectivity of $\frakB$. This completes the proof.
\end{proof}

\section{Global version}
\subsection{Reminders on the Lurie tensor product}
We follow Volpe's exposition~\cite{Volpe}. Let $\Mod_{\Mod(\bK)}(Cat_{comp})$ be the category of the cocomplete $\bK$-linear categories and the morphisms are $\bK$-linear cocontinuous functors. For $\bK$-linear  cocomplete categories $\cC, \cD$, there exists a $\bK$-linear cocomplete category $\cC\otimes_L \cD$ with a functor $\cC\times \cD\rightarrow \cC\otimes_L \cD$ satisfying 
\begin{equation}
    Fun_{!\times !}(\cC\times \cD, \cE)\cong Fun_{!}(\cC\otimes_L \cD, \cE)
\end{equation}
where the left hand side $Fun_{!\times !}$ denotes the $\bK$-linear functors preserving variable-wise colimits and the right hand side $Fun_!$ denotes the $\bK$-linear functors preserving colimits. The resulting category $\cC\otimes_L \cD$ is called Lurie's tensor product. In the following, we simply denote $\otimes_L$ by $\otimes$.

We will mainly use the following properties.
\begin{lemma}[{\cite[Corollary 2.30]{Volpe}}]\label{lem:Lurietensor}Let $X$ and $Y$ be manifolds. 
    \begin{enumerate}
        \item We have an equivalence. $\Sh(X,\bK)\otimes\Sh(Y,\bK)\cong \Sh(X\times Y,\bK)$.
        \item Let $\cC$ be a presentable $\bK$-linear category. Then we have $\Sh(X,\bK)\otimes\cC\cong \Sh(X,\cC)$.
    \end{enumerate}
\end{lemma}

\subsection{Tamarkin-type category}Let $X$ be a manifold. We consider the category $\Sh^{\bG}(X\times\bR_t,\bK)$ of the equivariant $\bK$-module sheaves on $X\times \bR_t$ with respect to the discrete $\bG$-action on the right component. Again, microsupport for each object of $\Sh^{\bG}(X\times\bR_t,\bK)$ is defined through the forgetful functor forgetting the equivariancy.

We denote the subcategory spanned by the object whose microsupport contained in $T^*X\times \bR\times \bR_{\leq 0}$ by $\Sh^{\bG}_{\leq 0}(X\times \bR_t,\bK)$. We set
\begin{equation}
   \mu^{\bG}(T^*X):= \Sh^{\bG}_{>0}(X\times \bR_t,\bK):=\Sh^{\bG}(X\times \bR_t,\bK)/\Sh^{\bG}_{\leq 0}(X\times \bR_t,\bK).
\end{equation}
Then $\mu^\bG(T^*X)$ is defined over $\Mod(L_0^\bG)$.

The following is known (cf. \cite{IK, KSZ}):
\begin{lemma}[Lurie tensor product]\label{lem:lurie}
    \begin{equation}
        \mu^\bG(T^*X)\cong \Sh(X,\bK)\otimes \mu^\bG(*).
    \end{equation}
\end{lemma}
We then have the following:

\begin{corollary}\label{cor:main}
    We have an almost embedding:
    \begin{equation}
    \mu^\bG(T^*X)\hookrightarrow_\almost \Sh(X,L_0^\bG).
    \end{equation}
    This functor will also be called $\frakA$ in the following.
\end{corollary}
\begin{proof}
Since $\Mod_\comp(L_0^\bG)\hookrightarrow \Mod(L_0^\bG)$ is a right adjoint, it induces
\begin{equation}
     \Sh(X,\Mod_\comp(L_0^\bG))\hookrightarrow  \Sh(X,\Mod(L_0^\bG)).
\end{equation}
Since $X$ is a manifold (in particular, hypercomplete), the category $ \Sh(X,\Mod(L_0^\bG))$ gets identified with the derived category $\Sh(X, L_0^\bG)$ of sheaves on $X$ whose values are $L_0^\bG$-modules. Hence, to finish the proof, it is enough to show $\mu^\bG(T^*X)\cong_\almost  \Sh(X,\Mod_\comp(L_0^\bG))$.

By Lemma~\ref{lem:Lurietensor} and Lemma~\ref{lem:lurie} and the presentability of $\mu^\bG(*)$, we have $\mu^\bG(T^*X)\cong \Sh(X, \mu^\bG(*))$. So, the problem is reduced to showing that $\Sh(X, \mu^\bG(*))\cong_\almost \Sh(X,\Mod_\comp(L_0^\bG))$. 

The functor $\frakA$ from Theorem~\ref{thm:main}, induces a functor $\Sh(X, \mu^\bG(*))\rightarrow \Sh(X,\Mod_\comp(L_0^\bG))$, which will also be denoted by $\frakA$. Similarly, we consider $\frakB\colon \Sh(X,\Mod_\comp(L_0^\bG))\to \Sh(X, \mu^\bG(*))$. For any $\cE\in \Sh(X,\Mod_\comp(L_0^\bG))$, the counit morphism $\cE\to \frakA\circ \frakB(\cE)$ is an almost isomorphism. Hence $\frakA$ is almost essentially surjective. For $\cE, \cF\in \Sh(X,\Mod_\comp(L_0^\bG))$,
\begin{equation}
    \Hom(\frakA(\cE), \frakA(\cF))\cong \Hom(\frakB\circ \frakA(\cE),\cF).
\end{equation}
Since $\frakB\circ \frakA(\cE)\cong_\almost \cE$, we conclude $\Hom(\frakB\circ \frakA(\cE),\cF)\cong_\almost\Hom(\cE,\cF)$. This shows the almost fully faithfulness of $\frakA$. 
\end{proof}

\subsection{Sheaf quantizations}
Let $X$ be a manifold. We denote the cotangent coordinate of $\bR_t$ by $\tau$.
We set 
\begin{equation}
    \lc \tau>0\rc:=\lc (p, (t, \tau))\in T^*X\times T^*\bR_t\relmid \tau>0\rc.
\end{equation}
It is known that $\SS(\cE)\cap \{\tau>0\}$ is well-defined for $\cE\in \mu^\bG(T^*X)$ where $\SS(\cE)$ is the microsupport of the underlying sheaf. We set
\begin{equation}
    \begin{split}
        \rho&\colon \lc \tau>0\rc \rightarrow T^*X; (p,t,\tau)\mapsto \tau^{-1}p\\
        &\musupp(\cE):=\text{the closure of }\rho(\SS(\cE)\cap \lc\tau>0\rc).
    \end{split}
\end{equation}

\begin{definition}[Sheaf quantization]
An object of $\mu^\bG(T^*X)$ is a sheaf quantization of a Lagrangian submanifold $L$ if 
\begin{enumerate}
    \item $\musupp(\cE)=L$, and
    \item the microstalks are finite dimensional.
\end{enumerate}
\end{definition}
For the construction and properties of sheaf quantizations, see the companion paper~\cite{IK}.
We denote the full subcategory of $\mu^\bG(T^*X)$ consisting of sheaf quantizations of projection-finite end-conic Lagrangians by $\SQ^\bG(T^*X)$.

\begin{corollary}\label{lem:embeddingSQ}
    We have an almost embedding
    \begin{equation}
        \SQ^\bG(T^*X)\hookrightarrow_\almost \Sh(X, L_0^\bG).
    \end{equation}
\end{corollary}

\subsection{Variant 1: Energy cutoff}
We sometimes would like to discuss the energy cutoff setup.

Let $X$ be a manifold and $\bR_{s<c}:=(-\infty,c)$ for $c>0$. Then we run the above theory to get $\mu^\bG(T^*X\times T^*\bR_{s<c})$. 
We consider the subcategory $\mu^\bG_{<c}(T^*X)$ consisting of the objects satisfying 
\begin{equation}
\begin{split}
    \SS(\cE)\cap \lc \tau>0\rc\subset & \lc (p, s, 0)\in (T^*X\times T^*\bR_t)\times T^*\bR_{s<c}\relmid s\geq 0\rc \\
    &\cup \lc (p', t, \tau, s, \sigma)\in T^*X\times T^*\bR_t\times T^*\bR_{s<c}\relmid s\geq 0, \tau=-\sigma\rc \\
    &\cup \lc (p',t,\tau, s, \sigma) \relmid s= 0, 0 \leq\tau \leq -\sigma\rc.
\end{split}
\end{equation}
We have \begin{equation}
    \mu^\bG_{<c}(T^*X)\cong \Sh(X,\bK)\otimes \mu^\bG_{<c}(*),
\end{equation}
see \cite{IK, KSZ}.

This category is a natural living place of sheaf quantizations of ``obstructed" Lagrangian branes. Those are realized as doubling movie sheaf quantization of Lagrangian; objects satisfying 
\begin{equation}
\begin{split}
    \SS(\cE)\cap \lc \tau>0\rc\subset &\lc (p, s, 0)\in (T^*X\times T^*\bR_t)\times T^*\bR_{s<c}\relmid p\in A, s\geq 0\rc \\
    &\cup \lc (p', t, \tau, s, \sigma)\in T^*X\times T^*\bR_t\times T^*\bR_{s<c}\relmid (p', t-s, \tau)\in A, s\geq 0, \tau=-\sigma\rc\\
    &\cup \lc (p',t,\tau, s, \sigma) \relmid s= 0, 0 \leq\tau \leq -\sigma\rc.
\end{split}
\end{equation}
for some Lagrangian subset $A\subset T^*X\times \bR_t\times \bR_{\tau>0}$. As obstructed Lagrangian Floer complex is defined over $\Lambda_0/T^c\Lambda_0$, the parallel story on the sheaf side can be discussed as a variant of the main story of this paper:
\begin{lemma}
    \begin{equation}
        \frakB_c\colon \Mod(L_0^\bG/T^cL_0^\bG)\hookrightarrow_\almost \mu^\bG_{<c}(*) 
    \end{equation}
    The right adjoint is given by $\frakA_c:=\Hom(\bigoplus_{c'\in \bR/\bG}T_{c'} 1_{\mu,c},-)$. We also have $\frakA_c\circ \frakB_c(M)\cong_\almost M$ for any $M$.
\end{lemma}
\begin{proof}
We set 
\begin{equation}
    \mu^\bG_{<c}(*)\ni 1_{\mu, c}:=\bigoplus_{c'\in \bG}\bK_{\lc (s,t)\relmid 0<s<c, c'\leq t< s+c' \rc}.
\end{equation}
By mimicking the argument in \cite{WKBkuw} and Lemma~\ref{lem:almostisom1}, one can see that the endmorphism is almost isomorphic to $L_0^\bG/T^cL_0^\bG$.
We then have a functor
\begin{equation}
    \Mod(L_0^\bG/T^cL_0^\bG)\rightarrow \mu^\bG_{<c}(*) ; M\mapsto M\otimes 1_{\mu,c}.
\end{equation}
Then this is almost fully faithful. The claim about the adjoint can be proved in the same way as Proposition~\ref{prop:inverse}.
\end{proof}

\begin{remark}
    The functor is not essentially surjective. We give an example which are not in the essential image. Consider the following nonzero infinite complex 
    \begin{equation}
        \cdots\xrightarrow{T^{1/2}}1_{\mu,1}\xrightarrow{T^{1/2}}1_{\mu,1}\xrightarrow{T^{1/2}}1_{\mu,1}\xrightarrow{T^{1/2}}\cdots  \end{equation}
One can see that this is right orthogonal to $1_{\mu,1}$. Indeed, if one apply $\Hom(1_{\mu,1},-)$ to the sequence the resulting sequence is
    \begin{equation}
        \cdots\xrightarrow{T^{1/2}}\Lambda_0/T\Lambda_0\xrightarrow{T^{1/2}}\Lambda_0/T\Lambda_0\xrightarrow{T^{1/2}}\Lambda_0/T\Lambda_0\xrightarrow{T^{1/2}}\cdots,\end{equation}
        which is acyclic. Hence $1_{\mu,1}$ does not generate the whole category. The author would like to thank T. Asano for asking a related question on the draft version.
\end{remark}

We immediately have the following.
\begin{corollary}
    \begin{equation}
        \mu^\bG_{<c}(T^*X)\hookleftarrow_\almost \Sh(X,\Mod(L_0^\bG/T^cL_0^\bG)).
    \end{equation}
    In particular, if $\bG=\bR$, we have
    \begin{equation}
        \mu^\bR_{<c}(T^*X) \hookleftarrow_\almost  \Sh(X, \Lambda_0/T^c\Lambda_0).
    \end{equation}
\end{corollary}
\begin{proof}
    For any $\cE, \cF\in\Sh(X,\Mod(L_0^\bG/T^cL_0^\bG))$, we have 
    \begin{equation}
        \Hom(\frakB_c(\cE), \frakB_c(\cF))\cong \Hom(\cE, \frakA_c\circ \frakB_c(\cF)).
    \end{equation}
    The RHS is almost isomorphic to $\Hom(\cE,\cF)$, since $\frakA_c\circ \frakB_c(\cF)\cong_\almost \cF$.
\end{proof}

\subsection{Variant 2: Higher-dimensional version}
For the use of $\hbar$-Riemann--Hilbert correspondence, we would like to mention the following version:

For a proper convex cone $\gamma$ in $\bR^n$, we say that $\gamma$ is simplicial if there exists a linear isomorphism of $\bR^n$ which gives an isomorphism $\gamma\cong \bR^n_{\geq 0}$.

Let $\gamma$ be a simplicial closed polyhedral proper convex cone in $\bR^n$ with nonempty interior. Then $\gamma$ has a semigroup structure with respect to the addition. We denote the corresponding polynomial ring by $\bK[\gamma]$. We denote the indeterminate corresponding to $a\in \gamma$ by $T^a$. Let $|\cdot|$ be the Euclidean norm of $\bR^n$. For $r\in \bR_{>0}$, we denote the ideal of $\bK[\gamma]$ generated by $T^a$'s with $|a|>r$ by $\frakm(r)$. We set
\begin{equation}
    \Lambda_0^{\gamma}:=\lim_{\substack{\longleftarrow \\ r\rightarrow \infty}}\bK[\gamma\cap \bG]/\frakm(r).
\end{equation}

We consider $\bR^n$-equivariant sheaves on $X\times \bR^n$. We denote the subcategory spanned by the object whose microsupport contained in $T^*X\times \bR^n\times (\bR^n\bs\Int(\gamma^{\vee}))$ by $\Sh^{\bR^n}_{\bR^n\bs \Int(\gamma^\vee)}(X\times \bR^n,\bK)$ where $\vee$ is polar dual.
We set
\begin{equation}
   \mu^{\gamma}(T^*X):= \Sh^{\bR^n}_{\Int(\gamma^\vee)}(X\times \bR^n,\bK):=\Sh^{\bR^n}(X\times \bR^n,\bK)/\Sh^{\bR^n}_{(\bR^n\bs \Int(\gamma^\vee))}(X\times \bR^n,\bK).
\end{equation}

By tensoring Theorem~\ref{cor:main} $n$ times, we get the following:
\begin{theorem}
Suppose $\gamma$ is simplicial.
    We have an almost embedding
    \begin{equation}
        \mu^{\gamma}(T^*X)\hookrightarrow_\almost \Sh(X, \Lambda_0^{\gamma}).
    \end{equation}
\end{theorem}

\begin{remark}
If $n=2$, every $\gamma$ is simplicial. It should be possible to remove the restriction on $\gamma$ and $\bG$. But, so far, we don't know any application of such a generalization.
\end{remark}

\section{Application I: Non-conic microlocal sheaf theory}
By Theorem~\ref{cor:main}, one can imagine that microlocal sheaf theory over Novikov rings are non-conic. In this section, we develop such a theory.

\subsection{Non-conic microsupport}
Let $R$ be a real valuation ring. Our examples are $\Lambda_0^\bG$ for a dense subgroup $\bG\subset \bR$. We denote the valuation by $v$. We set
\begin{equation}
    R_c:=R/\lc r\in R\relmid v(r)>c\rc.
\end{equation}

Let $U$ be an open subset of $X$. 
Let $\phi$ be a continuous function on $U$ which is bounded below. For any connected open subset $V\subset U$, we denote the infimum value of $\phi$ by $\phi_V$. We define a sheaf $R^\phi$ as follows: For any connected open subset $V$, we set $R^\phi(V):=R$. For an open inclusion of connected open subsets $W\subset V$, we have $\phi_W\geq \phi_V$, we set a structure morphism $R^\phi(V)\rightarrow R^\phi(W)$ by $T^{\phi_W-\phi_V}$. Sheafifying this, we get a sheaf on $U$.
We also set $R^\phi_c:=R^\phi\otimes_{R}R_c$ for $c\geq 0$.

\begin{definition}
For $\cE\in \Sh(X,R)$, we set $\musupp(\cE)$ to be the closure of the complement of the following set:
\begin{equation}
    \lc(x, \xi)\in T^*X\relmid \cHom(R^\phi_c,\cE)_x\simeq 0 \text{ for any $c\geq 0$ and $C^1$-function $\phi$ with $d\phi(x)=\xi$}\rc.
    \end{equation}
\end{definition}

The followings are obvious from the definition.
\begin{lemma}
\begin{enumerate}
    \item     $\musupp(\cE)$ is closed.
    \item Let $\cE_1\rightarrow \cE_2\rightarrow \cE_3\xrightarrow{[1]}$ be an exact triangle. Then we have
    \begin{equation}
        \musupp(\cE_2)\subset \musupp(\cE_1)\cup \musupp(\cE_3).
    \end{equation}
\end{enumerate}
\end{lemma}

Although we do not address here, generalizing functoriality results of microsupport in \cite{KS} to our setup should be an interesting problem.

\subsection{Relation to usual microsupport}
Since an object of $\Sh(X,R)$ is a sheaf on a manifold, we can also define the usual microsupport of \cite{KS}.
\begin{definition}
    For a subset $A\subset T^*X$, we set
    \begin{equation}
        \bR_{\geq 0}\cdot A:=\lc (x, \xi)\in T^*X\relmid (x, \xi')\in A, c\in \bR_{\geq 0} \text{ s.t.} (x, \xi)= (x, c\xi')\rc
    \end{equation}
\end{definition}
Note that the usual microsupport does not have much information for our sheaves, as the following proposition suggests. 
\begin{proposition}
For $\cE\in \Sh(X,R)$, we have $\bR_{>0}\cdot \musupp(\cE)\subset \SS(\cE)$.
\end{proposition}

\subsection{Relation to non-conic microsupport}
\begin{proposition}
    For an object $\cE\in \mu^\bG(T^*X)$, we have
    \begin{equation}
        \musupp(\cE)=\musupp(\frakA(\cE)).
    \end{equation}
\end{proposition}
\begin{proof}
Let $(x,\xi)$ be a point in $T^*X$. Consider any $C^1$-function with $\phi$ with $d\phi(x)=\xi$. For any $c'>0$, the equivariant sheaf $\bigoplus_{c\in \bR}\bK_{-\phi+c+c'>t\geq -\phi+c}$ is almost sent to $R_{c'}^\phi$ on sufficiently small open subset around $x$ under $\frakA$. Hence the test sheaves to estimate the both sides of the desired equality coincide. This completes the proof.
\end{proof}

\section{Application II: Curved sheaves}
In Fukaya category theory, we have to deal with curved complexes, bounding cochains, and bulk deformations. In sheaf theory (in the setup of Tamarkin category), introducing such notions is not easy (although we can do them as partly explained in \cite{IK}).
Our interpretation of the category $\Sh^\bR_{\tau>0}(X\times \bR_t)$ as the sheaf category of $\Lambda_0$-modules allows us to introduce such deformations easily.

\subsection{Curved complex and sheaves}
We set $R:=\Lambda_0$. Let $V$ be a $\bZ$-graded $R$-module and $d$ be a $\deg=1$-endomorphism of $V$. We call such a pair $\bV:=(V, d)$, a curved $R$-module complex. We sometimes use the notation $d_\bV:=d$. We call $d^2$ is the curvature of $\bV$. 

A morphism between a curved complex is a graded $R$-linear morphism between underlying $\bZ$-graded $R$-modules. We denote the category of curved $R$-module complexes by $CCh^c(R)$.

Let $\bV_i:=(V_i, d_i)$ ($i=1,2$) be curved complexes. The tensor product is defined by the graded tensor product $V_1\otimes V_2$ equipped with the differential $d_{\bV_1}\otimes 1+1\otimes d_{\bV_2}$. This defines a monoidal category structure.
\begin{definition}
    We say a category enriched over $CCh^c(R)$ is an $R$-linear curved dg-category.
\end{definition}

Let $(V_i, d_i)$ be curved complexes. The space of linear maps $\Hom(V_1, V_2)$ is again a curved complex $\bH om(V_1, V_2)$ where its differential is defined by
\begin{equation}
    f\mapsto d_{V_2}\circ f-f\circ d_{V_1}.
\end{equation}
Hence $CCh^c(R)$ is itself a curved dg-category.

\begin{definition}[Flat part]
    Let $\bV$ be a curved complex. Note that $\bV_0:=(\ker(d^2), d)$ is a usual complex. We say $\bV_0$ is the flat part of $\bV$.
\end{definition}

Let $(V_i, d_i)$ ($i=1,2$) be curved complexes. Considering $\bH om(V_1, V_2)_0$, we obtain the dg-category $CCh(R)$ of curved complexes. This category contains the dg-category of chain complexes of $R$-modules $Ch(R)$ as a subcategory.

\begin{definition}
Let us consider a sequence of curved complexes:
\begin{equation}
    \bV_1\rightarrow \bV_2\rightarrow \cdots \bV_n.
\end{equation}
Suppose that this sequence is an exact sequence for each graded part. Such a sequence is called an acyclic complex.
\end{definition}

Now we construct the so called totalization complex. For the definition, we refer to \cite{Positselski}.

\begin{remark}
In \cite{Positselski}, there are three kinds of derived categories. We choose the one caled ``absolute" one. Since the objects we are interested in are finite in some sense, we believe that this choice is not essential for our purpose.
\end{remark}

We denote the full subcategory spanned by the totalizations of the acyclic complexes by $Acycl$, We set the Drinfeld quotient by 
\begin{equation}
    \mathrm{CCh}(R):=CCh(R)/Acycl.
\end{equation}

Similarly, we consider the category of curved sheaves as a global version of the above story. Let $X$ be a manifold. A curved sheaf $\cE$ is a $\bZ$-graded sheaf with a degree 1 endomorphism. We consider the category defined by
\begin{enumerate}
    \item the objects is the curved sheaves
    \item a morphism is a graded morphism between graded sheaves underlying curved sheaves.
\end{enumerate}
We denote this category by $CSh^c(R_X)$, which is a curved dg-category. By replacing the hom-spaces by flat parts, we obtain a dg-category $CSh(R_X)$. We can similarly define the subcategory of the totalization of the acyclic complexes $Acycl$. Then we define
\begin{equation}
    \mathrm{CSh}(X, R):=CSh(R_X)/Acycl.
\end{equation}
This is the derived dg category of curved sheaves. 

\subsection{Twisted sheaves}
Inside $\mathrm{CSh}(X, R)$, there exists a well-behaved subclass of objects: weakly unobstructed sheaves. For any object $\cE\in \CSh(X,R)$, we have $\cE\otimes R_X\cong \cE$ under derived tensor product. This implies that we have a morphism
\begin{equation}
    w\colon \cE\rightarrow \cE[2]
\end{equation}
associated to each $w\in H^2(X,R)$.
\begin{definition}
Take $w \in H^2(X,R)$. We denote the full subcategory of $\CSh(X,R)$ consisting of the objects whose curvature is cohomologically $w$ by $\Sh(X, R, w)$.
\end{definition}
We first note that we have the following isomorphism:
\begin{equation}
   e^{(-)} \colon \frakm\xleftrightarrow{\cong} 
 1+\frakm \colon \log.
\end{equation}

Recall that $\frakm$ is the maximal ideal of $R$.
Take a cohomology $e^w\in H^2(X, 1+\frakm)$. Fix a Cech 2-cocycle $e^{c_{ijk}}\in 1+\frakm$. We consider the category consisting of objects as follows:
\begin{enumerate}
    \item For each $U_i$, we have an object $\cE_i$ of $\Sh(U_i, R)$,
    \item On the restriction to $U_i\cap U_j$, we have  a specified isomorphism $\cE_i\cong \cE_j$ in $\Sh(U_i\cap U_j, R)$.
    \item On the restriction to $U_i\cap U_j\cap U_k$, the associated automorphism of $\cE_i$ is $e^{c_{ijk}}$.
\end{enumerate}
The resulting category does not depend on the choice of Cech representative of an element $e^w\in H^2(X,1+\frakm)$. We denote the resulting category by $\Sh_{tw}(X, R, e^w)$.

We now deduce that the above two categories are just two presentations of the same category.
\begin{theorem}
For $w\in H^2(X, \frakm)$,
we have $\Sh_{tw}(X, R, e^w)\simeq \Sh(X, R, w)$.
\end{theorem}
\begin{proof}
For the sake of simplicity, we assume $\bK=\bC$. Then we can use the de Rham model for the cohomology e.g. $H^2(X,\frakm)\subset \prod_{c\in \bR_{>0}} H^2(X, \bC)\cong \prod_{c\in \bR_{>0}} H^2_{dR}(X, \bC)$.  

Take a good cover $\{U_i\}$. Over each $U_i$, the restriction $w|_{U_i}$ has a primitive 1-form $\alpha_i$. 

Take an object $\cE\in \Sh(X, R, w)$. Twisting by $\alpha_i$ gives an equivalence of $\cE|_{U_i}\in \Sh(U_i, R, w|_{U_i})\cong \Sh(U_i, R)\ni \cE_i$. On the overlap $U_i\cap U_j$, we have an isomorphism given by 
\begin{equation}
    \cE_i\xrightarrow{\times e^{f_{ij}}} \cE_j
\end{equation}
where $f_{ij}\in C^\infty(U_{i}\cap U_j, \Lambda_+)$ is a primitive of $\alpha_i-\alpha_j$. As in the usual Cech--de Rham isomorphism, the composition $e^{f_{ij}}e^{f_{jk}}e^{f_{ki}}$ is a constant and given by $e^{c_{ijk}}$ where $c_{ijk}$ is a Cech representative of $w$. This completes the proof.  

For general coefficients, the same proof works by replacing the de Rham resolution with Cech resolution.
\end{proof}

\begin{example}[Curved connection]
There exists a category closely related to the above idea. 
Note that, over the field $\bC$, the isomorphism $e^{(-)}$ is extended to 
\begin{equation}
    e^{(-)}\colon \Lambda_0\rightarrow \bC^*+\frakm
\end{equation}
where the RHS is the units of $\Lambda_0$. 

Let $\cE$ be a $C^\infty$-module with a flat connection. If $\cE$ is associated to a vector bundle, it is well-known that the flat sections form a locally constant sheaf and the assignment gives an equivalence.

Similarly, if $\cE$ is $C^\infty$-module with a connection whose curvature is $w\in \Omega^2(X,R)$, they form a dg-category. Then, the above theorem tells us that the category of vector bundles with connections such that the curvature $w$ can be embedded into the category of twisted sheaves.

This construction should be related to the B-field deformation/bulk deformation of Fukaya category.
\end{example}

\subsection{Twisted sheaf quantization and bounding cochain}
In this section, we explain how we can run the theory in \cite{IK} in the twisted setup.
For the details, we refer to \cite{IK}.

Note that any object in $\Sh(X,R,w)$ can be locally viewed as an object of $\Sh(X,R)$, one can define $\musupp$. Similarly, we say an object $\cE$ in $\Sh(X,R,w)$ is a sheaf quantization if it is locally a sheaf quantization viewed as an object of $\Sh(X,R)$.

Also, the low-energy standard sheaf quantization construction has local nature, we get a low-energy sheaf quantization in $\Sh(X,R,w)$ for any Lagrangian brane. Then, in the exactly the same way, one can construct a curved dga associated to a Lagrangian brane. An existence of a Maurer--Cartan element implies an existence of sheaf quantization in $\Sh(X,R,w)$.

\begin{remark}
    It is even possible to formulate curved sheaf quantizations: namely, instead of considering curved twisted complex of hom-spaces as in \cite{IK}, one can directly construct a curved twisted complex of sheaves in $\CSh(X, R)$. Then one can formulate the Maurer--Cartan equation in \cite{IK} as a Maurer--Cartan equation of the curved sheaf itself. The resulting theory is obviously the same one obtained in \cite{IK}.
\end{remark}

\footnotesize
\bibliographystyle{alpha}
\bibliography{bibs.bib}

\noindent
Tatsuki Kuwagaki\\
Department of Mathematics, Kyoto University \\
Kitashirakawa Oiwakecho, Sakyo-ku, Kyoto 606-8502, Japan\\
Email: {tatsuki.kuwagaki.a.gmail.com}\\
\end{document}